\documentclass[11pt,a4paper,reqno]{amsart}

\usepackage{amssymb}
\usepackage{color}

\renewcommand{\epsilon}{\varepsilon}

\newtheorem{theorem}{Theorem}[section]
\newtheorem{proposition}[theorem]{Proposition}
\newtheorem{corollary}[theorem]{Corollary}
\newtheorem{lemma}[theorem]{Lemma}

\theoremstyle{definition}

\newtheorem{definition}[theorem]{Definition}

\theoremstyle{remark}
\newtheorem{remark}[theorem]{Remark}

\newcommand{\pd}{\operatorname{pd}}
\newcommand{\cd}{\operatorname{cd}}

\newcommand{\vcd}{\operatorname{vcd}}

\newcommand{\Z}{\mathbb Z}

\newcommand{\SL}{\operatorname{SL}}

\newcommand{\cohom}[3]{H^{{\raise1pt\hbox{$\scriptstyle#1$}}}(#2\>\!,#3)}
\newcommand{\tatecohom}[3]%
  {\widehat H^{{\raise1pt\hbox{$\scriptstyle#1$}}}(#2\>\!,#3)}

\newcommand{\Cohom}[3]%
  {H^{{\raise1pt\hbox{$\scriptstyle#1$}}}\big(#2\>\!,#3\big)}
\newcommand{\Tatecohom}[3]%
  {\widehat H^{{\raise1pt\hbox{$\scriptstyle#1$}}}\big(#2\>\!,#3\big)}

\newcommand{\homol}[3]{H_{{\lower1pt\hbox{$\scriptstyle#1$}}}(#2\>\!,#3)}
\newcommand{\homolog}[2]{H_{{\lower1pt\hbox{$\scriptstyle#1$}}}(#2)}

\newcommand{\e}{{\underline{\text {E}}}}
\newcommand{\EG}{{\text{E}G}}

\usepackage[normalem]{ulem}

\title{The proper geometric dimension of the mapping class group}

\author{Javier Aramayona and Conchita~Mart\'inez-P\'erez}

\address{Javier Aramayona \\ School of Mathematics, National University of Ireland Galway.
} \email{\tt javier.aramayona@nuigalway.ie}

\address{Conchita Mart\'inez-P\'erez, Departamento de Matem\'aticas, Universidad de Zaragoza,
50009 Zaragoza, Spain} \email{conmar@unizar.es}

\date{\today} 
\keywords{}
\subjclass[2000]{
20J05}

\thanks{Supported by Gobierno de Arag\'on, European Regional Development Funds and
MTM2010-19938-C03-03. }

\begin{document}

\thispagestyle{empty}

\begin{abstract} We show that the mapping class group of a closed surface admits a cocompact classifying space for proper actions
of dimension equal to its virtual cohomological dimension.  
\end{abstract}

\maketitle

\section{Introduction}

Let $\Gamma_{g,n}$ be the mapping class group of a connected orientable surface of genus $g$ with $n$ marked points. In this note we are interested in the minimal dimension $\underline{\rm{gd}}(\Gamma_{g,n})$ of a  classifying space $\e \Gamma_{g,n}$ for proper actions of $\Gamma_{g,n}$. Recall that, given a discrete group $G$, the space $\e G$ is a contractible space on which $G$ acts properly, and such that the fixed point set of a subgroup $H<G$ is contractible if $H$ is finite, and is empty otherwise. 

Since $\Gamma_{g,n}$ is virtually torsion-free, its virtual cohomological dimension $\vcd(\Gamma_{g,n})$ is a lower bound for $\underline{\rm{gd}}(\Gamma_{g,n})$; we remark, however, that there are groups for which the inequality is strict \cite{ian-brita}. In \cite{harer}, Harer computed  $\vcd(\Gamma_{g,n})$ for all $g,n\ge 0$, see Theorem \ref{harer} below. A central ingredient of Harer's argument is the construction, for $n>0$, of a  cocompact $\Gamma_{g,n}$-equivariant deformation retract (a {\em spine}) of Teichm\"uller space $\mathcal T_{g,n}$, which happens to have dimension $\vcd(\Gamma_{g,n})$ when $n=1$.  Hensel, Osajda, and Przytycki \cite{HOP} have proved that Harer's spine is in fact an $\e \Gamma_{g,n}$ for all $g,n$; in particular, $\vcd(\Gamma_{g,1}) = \underline{\rm{gd}}(\Gamma_{g,1})$.

On the other hand, the case of closed surfaces of genus $g\ge 2$ is far from well-understood.  Ji and Wolpert \cite{ji-wolpert} used the fact that the Teichm\"uller space $\mathcal T_{g,n}$ is an $\e \Gamma_{g,n}$ to prove that the {\em thick part} $\mathcal T^{\ge \epsilon}_{g,n}$ of  Teichm\"uller space is a cocompact  $\e \Gamma_{g,n}$ for all $g,n\ge 0$ (see also Broughton \cite{broughton} and Mislin \cite{mislin} for an alternative construction). Ji \cite{ji} has recently exhibited cocompact spines of $\mathcal T_{g,0}$ of dimension less than $\dim(\mathcal T_{g,0}) $ -- but also greater than $\vcd(\Gamma_{g,0})$ when $g>2$-- that serve as cocompact models of $\e \Gamma_{g,0}$. However, it is not known whether there is a cocompact spine of $\mathcal T_{g,0}$ of optimal dimension; see Question 1.1 of \cite{farb}. More generally, Bridson and Vogtmann have asked whether if it possible to construct a cocompact $\e \Gamma_{g,0}$ of dimension equal to $\vcd(\Gamma_{g,0})$; see Question 2.3 of \cite{farb}. The purpose of this note is to  prove the existence of such $\e \Gamma_{g,0}$: 

\begin{theorem}\label{main}
For any $g\ge 0$ there exists a cocompact  $\underline{\rm{E}} \Gamma_{g,0}$ of dimension equal to $\vcd(\Gamma_{g,0})$. In other words,  $\underline{\rm{gd}}(\Gamma_{g,0}) = \vcd(\Gamma_{g,0})$. 
\end{theorem}

The main tool of our proof is the algebraic invariant  $\underline{\cd}(\Gamma_{g,n})$ (see section \ref{prelimbredon}) which serves as the algebraic counterpart of $\underline{\rm{gd}}(\Gamma_{g,n})$. These two invariants are related  in the same way as the ordinary cohomological dimension of a group $G$ is related to the minimal dimension of an $\EG$. For example, generalizing what happens in the torsion-free case, L\"uck \cite{luck} proved the following Eilenberg-Ganea-type theorem, which will play a central role in our proof:

\begin{theorem}[\cite{luck}]\label{luck}
Let $G$ be a group with  $\underline{\rm{cd}}(G)=d\geq 3$. Then there is a $d$-dimensional   $\underline{\rm{E}} G$. Moreover, if $G$ has a cocompact $\underline{\rm{E}} G$ then it
also admits a cocompact $\underline{\rm{E}} G$ of dimension $d$. 
\end{theorem}

 In the light of L\"uck's theorem, we will prove that  $\underline{\rm{cd}}(\Gamma_{g,0}) =\vcd ( \Gamma_{g,0})$ whenever $g\ge 3$, using a result of the second author stated as Theorem \ref{bredon} below. The case  $g\le 2$ will require separate treatment.

An inductive argument along the lines of Section 4 of \cite{mislin}, using the Birman short exact sequence, Theorem 5.16 of \cite{luck2} and Harer's formula yields the analog of Theorem \ref{main} for arbitrary surfaces:


\begin{corollary}
For all $g,n,b\ge 0$, there exists a cocompact  $\underline{\rm{E}} \Gamma^b_{g,n}$ of dimension equal to $\vcd(\Gamma^b_{g,n})$.
\label{cor:punctures}
\end{corollary}

 We stress that, in the case $b=0$ and $n=1$, one may alternatively make use of Harer's spine.

\begin{remark}
	In the published version of this article, one of the two arguments we gave for Corollary \ref{cor:punctures} made an incorrect use of Harer's spine, which is in fact only valid for $n=1$; we are grateful N. Colin, R. Jim\'enez-Rolland, P.L. Le\'on-\'Alvarez, and L.J. S\'anchez-Salda\~{n}a for pointing this out. Instead, we have included here the other argument, contained in Remark 4.5 of the published version, which we have embedded in the paragraph before Corollary \ref{cor:punctures} above. 
\end{remark}

\noindent{\bf Acknowledgements.} The authors would like to thank Chris Leininger, Greg McShane and Juan Souto for conversations.

\section{Preliminaries}

Let $S_{g,n}$ be a connected, orientable surface of genus $g\ge 0$, with empty boundary and $n\ge 0$ marked points. The mapping class group
$\Gamma_{g,n}$  is the group of isotopy classes of
	orientation preserving homeomorphisms of $S_{g,n}$ which map every marked point to itself; we remark that this group is often referred to in the literature as the ``pure'' mapping class group. For simplicity, we will  write $S_g:=S_{g,0}$ and $\Gamma_{g}:= \Gamma_{g,0}$.

\subsection{Virtual cohomological dimension}

Recall that $\Gamma_{g,n}$ has a torsion-free subgroup of finite index. As mentioned earlier, Harer \cite{harer} computed the virtual cohomological dimension $\vcd(\Gamma_{g,n})$ of $\Gamma_{g,n}$:

\begin{theorem}[Harer]\label{harer} If $2g+n>2$, then 
$$\vcd(\Gamma_{g,n}) = \left \{ \begin{aligned}
&4g+n-4\text{ if } g,n>0\\
&4g-5,\text{ if } n=0\\
&n-3, \text{ if } g=0 
\end{aligned} \right. $$
\end{theorem}

\begin{remark} 
If $n\le 1$ then  $\Gamma_{0,n}$ is trivial. Also, $\Gamma_{0,2} \cong \mathbb Z$, and $\Gamma_{1,0}\cong \SL_2(\mathbb Z)$ \cite{farb-margalit}. Therefore, $\vcd(\Gamma_{0,n}) = 0$ for $n\le 1$, and $\vcd(\Gamma_{0,2})=\vcd(\Gamma_{1,0})= 1$.
\end{remark}

 \subsection{Riemann-Hurwitz formula}
Let $g\ge 2$. By the Nielsen Realization Theorem \cite{kerckhoff}, every finite subgroup of $\Gamma_{g}$ may be realized as a group of isometries with respect to some hyperbolic metric on $S_{g}$. Therefore, given  a finite subgroup $L\le \Gamma_{g}$, and slightly abusing notation, we may consider the (hyperbolic) orbifold $S_{g}/L$. We denote by $g_L$ the genus of $S_{g}/L$; similarly, let $k_L$ be the number of orbifold points of $S_{g}/L$, of orders $p_1^L, \ldots, p_{k_L}^L$, respectively. The tuple $(g_L;p^L_1,\ldots,p^L_{k_L})$ is called the {\em signature} of $L$. Since $S_g \to S_g/L$ is an orbifold cover of degree $|L|$, the multiplicativity of the orbifold Euler characteristic implies that $g$ and the signature of $S_g/L$ are related by the so-called {\em Riemann-Hurwitz formula} -- see e.g. \cite{farb-margalit}: 
\begin{equation}\label{rh}
\frac{2g-2}{|L|}=2g_L-2+l_L,
\end{equation}
where 
\begin{equation}\label{eq:order}
l_L=\sum_{i=1}^{k_L}\left (1-\frac{1}{p^L_i}\right ).
\end{equation}
Observe that \eqref{eq:order} implies:
\begin{equation}\label{eq:order2}
\frac{k_L}{2} \le l_L \le k_L.
\end{equation}

We will need the following surely well-known observation:

\begin{lemma}
Let  $L<T$ be two distinct finite subgroups of $\Gamma_g$, with $g\ge 2$, and denote by $(g_L; q_1, \ldots, q_{k_L})$ and $(g_T; p_1, \ldots, p_{k_T})$  the signatures of $S/L$ and $S/T$, respectively. Then:

\begin{enumerate}
\item If $g_T >1$ then $g_T < g_L$;
\item If $g_T\le 1$ then $g_T \le g_L$; moreover, if $g_T=g_L$ then $k_T<k_L$.
\end{enumerate}
\label{L:gendec}
\end{lemma}

\begin{proof}
The map $S_g/L \to S_g/T$ is an orbifold cover of degree $d=[T:L]>1$. By the Riemann-Hurwitz formula, we have: 
\begin{equation}
2-2g_L +\sum_{i=1}^{k_L} 1/q_i - k_L= d\left ( 2-2g_T + \sum_{i=1}^{k_T} 1/p_i- k_T \right ). \nonumber
\end{equation}
Now $$ \sum_{i=1}^{k_L} 1/q_i= d\sum_{i=1}^{k_T} 1/p_i.$$ (see, for instance, Section 7.2.2 of \cite{farb-margalit}.) Hence
\begin{equation}
2g_L  + k_L -2 = d(2g_T-2) + dk_T.
\label{E:euler}
\end{equation}
We prove the first claim of the lemma. Assume that $g_T>1$ and suppose, for contradiction, that $g_L \le g_T$. From \eqref{E:euler}:
$$d(2g_T - 2) + dk_T= 2g_L - 2 + k_L\le 2 g_T - 2 +dk_T,$$
which implies that $d(2g_T - 2) \le 2 g_T - 2$;  a contradiction since $d>1$. We have thus proved part (i).

Moving on to the second claim of the lemma, assume $g_T \le 1$. We first prove that $g_T\le g_L$. Arguing again by contradiction, the only case to rule out is $g_T=1$ and $g_L=0$. From \eqref{E:euler}, we get:
$k_L-2 = dk_T \ge k_L$, which is impossible. Hence $g_L\le g_T$, as claimed.

Finally, we prove that $k_T < k_L$ whenever $g_T = g_L$; recall that the latter implies $g_T\in \{0,1\}$.  First, if $g_T=1$,  \eqref{E:euler} gives $k_L =  dk_T$, and thus $k_T <k_L$, as desired.  If $g_T=0$, again \eqref{E:euler} yields $k_L - 2=  d(k_T - 2)$, which gives
$k_T<k_L$ as well (observe that $k_L,k_T>2$ since $g_T = g_L=0$). This finishes the proof of the lemma.
\end{proof}

\section{Preliminaries on classifying spaces for proper actions} \label{prelimbredon}

As mentioned in the introduction, we will determine $\underline{\rm{gd}}(\Gamma_{g,n})$ using the algebraic invariant $\underline{\cd}(\Gamma_{g,n})$, which is defined along the lines as the ordinary cohomological dimension but in the setting of proper actions. Informally, it is the length of the shortest projective resolution of the trivial object in a certain category, whose objects are called Bredon modules.
 Here, we will only need to make use of two facts about $\underline{\cd}(G)$, referring the reader to \cite{luck, luck2} for a  discussion on $\underline{\cd}(G)$.
 
The first fact about $\underline{\cd}(G)$ that we will need is L\"uck's Theorem \ref{luck}, which is a consequence of Theorem 13.19 of \cite{luck}. A proof of the existence of the model in Theorem \ref{luck} was given by Brady-Leary-Nucinkis \cite{bln}; we now explain how to adapt their argument to produce a cocompact one.

Let $X$ be a cocompact $\e G$ of dimension $d$. 
The $(d-1)$--skeleton $Z$ of $X$ gives a chain complex of free Bredon modules, which is exact except possibly in degree $d-1$. Let  $M$ be the $(d-1)$--th homology group of $Z$. 
As in the classical case -- see Lemma 2.1 of  (\cite{brown}, Section VIII) -- $M$ is a projective Bredon module, which is finitely generated since $X$ is cocompact.  Note that $M$ might not be free. However, the versions for Bredon modules Lemma 4.4 and Proposition 6.5 of (\cite{brown}, Section VIII) together imply that there is some free Bredon module $P$ such that $F:=P\oplus M$ is free and that $P$ can be taken to be finitely generated, so that $F$ is also finitely generated. 
A finitely generated free Bredon module is determined by a finite family of representatives of conjugacy classes of finite subgroups. Denote by $\Lambda_P,\Lambda_F$, respectively, the families for $P$ and $F$.
Now, attach to $Z$ orbits of $(d-1)$--cells of types $S^{d-1}\times G/H$ for $H\in\Lambda_P$, and use Hurewicz's Theorem to attach orbits of $d$-cells of types $S^d\times G/H$ for $H\in\Lambda_F$. This way we get a new cocompact $CW$-complex $Y$ such that  the fixed point set of $L<G$ is contractible whenever $L$ is finite and empty otherwise.  In other words, $Y$ is the desired model for $\e G$.  

\begin{remark} Theorem \ref{luck} also holds if $d=1$; see \cite{bln}. \end{remark}

Before we describe the second property of $\underline{\cd}(G)$ that will be used, we need some definitions. Consider, for every finite subgroup $H<G$, the {\em Weyl group} $$WH:=N_{G}(H)/H,$$ where $N_{G}(H)$ denotes the normalizer of $H$ in $G$. 
Observe that the centralizer $Z_{G}(H)$ of $H$ has finite index in $N_{G}(H)$, and thus  $WH$ and $Z_{G}(H)$ are weakly commensurable. 
Let $\mathcal{F}_H=\{T\leq G\text{ finite }\mid H<T\}$, noting that the group $WH$ acts on the poset $\mathcal{F}_H$ by conjugation. 

Let ${\mathcal{F}}_{H\bullet}$ be the chain complex of $G$-modules associated to the geometric realization of $\mathcal{F}_H$, and let $\Sigma\widetilde{\mathcal{F}}_{H\bullet}$ be the result of augmenting and suspending $\mathcal{F}_{H\bullet}$.
Finally, write $\text{pd}_{WH}\Sigma\widetilde{\mathcal{F}}_{H\bullet}$ for the projective dimension of the chain complex $\Sigma\widetilde{\mathcal{F}}_{H\bullet}$, namely the shortest length of a chain complex $P_\bullet$ of projective $G$-modules such that there is a morphism $P_\bullet\to\Sigma\widetilde{\mathcal{F}}_{H\bullet}$ inducing an isomorphism in the homology groups.

A result of Connolly-Kozniewsky, stated as Theorem A in  \cite{conch2}, implies 
\begin{equation}\label{pd}\underline{\cd}(G)=\max_{H\leq G \text{ finite}}\text{pd}_{WH}\Sigma\widetilde{\mathcal{F}}_{H\bullet}.\end{equation}

We will need:

\begin{definition}[Length] The {\em length} $\lambda (L)$ of a finite group $L$ is the largest number $i\in \mathbb{N} \cup\{0\}$ for which there is a sequence
$1=L_0<L_1<\ldots<L_i=L.$
\end{definition}

We are finally ready to introduce the promised second fact about $\underline{\cd}$, which follows as an easy consequence of (\ref{pd}):

\begin{theorem}\label{bredon} Let $G$ be a virtually torsion-free group such that for any $H\leq G$ finite,
$\vcd(WH)+\lambda(H)\leq\vcd(G).$ Then $\underline{\cd}(G)=\vcd(G)$.
\end{theorem}

\begin{proof}  For $i\geq 0$, the $i$-th term of $\Sigma\widetilde{\mathcal{F}}_{H\bullet}$ is the permutation module associated to the action of $G$ on the cells of the form $T:=H_i>\ldots>H_1>H_0=H$, whose stabilizer is weakly commensurable with $WT$.  Observe that $i+\lambda(H)\leq \lambda(T)$. Therefore
$$\pd_{WH}\Sigma\widetilde{\mathcal{F}}_{H\bullet}\leq\text{max}\{ \lambda(T)-\lambda(H) +\vcd WT\mid T\in\mathcal{F}_H\cup\{H\}\}\leq\vcd(G).$$
Using (\ref{pd}) we get $\underline{\cd}(G)\leq\vcd(G)$; the other inequality is well known. \end{proof}

\section{Proof of Theorem \ref{main}}

In the light of Theorem \ref{bredon}, we are going to need to understand the relation between $\vcd(\Gamma_g)$ and $\vcd(WL)$, for every finite subgroup $L<\Gamma_g$.
The following is well-known; see, for instance, Proposition 2.3 of \cite{maher}:

\begin{lemma}\label{centralizer}
Let $L\leq\Gamma_g$ be a finite subgroup of signature $(g_L;p^L_1,\ldots, p^L_{k_L})$. Then $WL$ has finite index in $\Gamma_{g_L,k_L}$. In particular, $\vcd(WL) = \vcd(\Gamma_{g_L,k_L})$.
\end{lemma}

For notation purposes, it will be convenient to write $$\nu(L):=4g_L+k_L-4.$$ Observe that, from Theorem \ref{harer} and Proposition \ref{centralizer}, we have:
\begin{equation}\label{E:nu}\vcd(WL)=\left \{\begin{aligned}
&\nu(L), \text{ if }g_L,k_L>0\\
&\nu(L)-1, \text{ if }k_L=0\\
&\nu(L)+1, \text{ if }g_L=0\\
\end{aligned}\right .
\end{equation}

 We will need:

\begin{proposition}\label{cases} Let $L<T$ be finite subgroups of $\Gamma_g$, where $g\ge 2$. Assume that $g_T<g_L$.
Then $\vcd(WT)<\vcd(WL)$, unless we are in one of the following two cases:
\begin{itemize}

\item[(i)] $(g_L,k_L)=(2,0)$ and $(g_T,k_T)=(0,6)$.

\item[(ii)] $(g_L,k_L)=(1,r)$ and $(g_T,k_T)=(0,r+3)$, for some $r\ge 1$. 
\end{itemize}
\end{proposition}

\begin{proof}
First, observe that $L<T$  implies   $Z_{\Gamma_g}(T)\leq Z_{\Gamma_g}(L)$, and thus $\vcd(WT)\leq\vcd(WL)$.
Using \eqref{eq:order2} and the Riemann-Hurwitz formula we deduce:
\begin{eqnarray}\label{E:dichot}
 \nu(T) &=& 4g_T -4+k_T  \le 4g_T-4 +2l_T = \frac{4g_L -4+2l_L}{|T:L|}      \nonumber \\
   &\le& \frac{4g_L-4 +2k_L}{|T:L|} = \frac{\nu(L) +k_L}{|T:L|} \le \frac{\nu(L) +k_L}{2}.
\end{eqnarray}

Armed with inequality \eqref{E:dichot}, and noting that $g_L>0$, we distinguish the following cases:

\bigskip

\noindent{\bf CASE 1: $g_T>0$.} We have the following subcases:

\medskip

\noindent {\bf (1a: $k_T=k_L=0$.)} 
Since $g_T < g_L$ then $\vcd(WT) < \vcd(WL)$.

\smallskip

\noindent {\bf (1b: $k_L =0$, $k_T\ne0$)}. Since $g_L\ge 2$,  
 we have that $\nu(L) \ge 4$. From \eqref{E:nu}:
$$\vcd(WT)= \nu(T)\le \frac{\nu(L)}{2} <\nu(L) - 1=\vcd(WL).$$

\smallskip

\noindent {\bf (1c: $k_L\ne 0$, $k_T=0$)}. Note that $g_T\ge 2$,  and so $g_L\ge 3$. In particular, $\vcd(WL)= 4g_L - 4 +k_L \ge k_L +8$. Therefore, using \eqref{E:nu} and \eqref{E:dichot}:
$$\vcd(WT) = \nu(T) - 1 \le \frac{\nu(L) +k_L - 2}{2}<\vcd(WL).$$

\noindent {\bf (1d: $k_L \ne 0$, $k_T \ne 0$).} In this case, since $0<g_T < g_L$, then $\vcd(WT)<\vcd(WL)$.

\bigskip

\noindent{\bf CASE 2: $g_T =0$}. Note that $k_T >0$. We have the following subcases:

\medskip

\noindent {\bf (2a: $k_L =0$.)} Again by \eqref{rh}, $g_L\ge 2$, and in particular $\vcd(WL)\ge 3$. From this, and using \eqref{E:nu} and \eqref{E:dichot}, we deduce: $$\vcd(WT)-1 \le \frac{\vcd(WL)+1}{2} < \vcd(WL)-1$$
unless $\vcd(WL)=3$. In the latter case,  either $\vcd(WT)<\vcd(WL)$ or
$(g_L,k_L)=(2,0)$ and $(g_T,k_T)=(0,6)$, as claimed. 

\smallskip

\noindent {\bf (2b: $k_L \ne 0$).} Suppose first that $g_L\ge 2$, in which case $\vcd(WL) = 4g_L - 4 + k_L \ge k_L+4$. From \eqref{E:nu} and \eqref{E:dichot}, we obtain:
 $$\vcd(WT)-1\le \frac{\vcd(WL)+k_L}{2}\le \frac{2\vcd(WL)-4}{2}<\vcd(WL)-1,$$ and thus the result follows. Suppose now that  $g_L< 2$, and thus $g_L=1$ as $0=g_T<g_L$. As $\vcd WT\leq\vcd WL$, in the equality case we have $$k_L = \vcd(WL) = \vcd(WT) = k_T - 3,$$ and we are in part (ii) of the theorem.
\end{proof}

\begin{remark} Cases (i) and (ii) in Theorem  \ref{cases} do occur in practice. Indeed, there is a branched double-cover $S_{2,0} \to S_{0,6}^*$, where $S_{0,6}^*$ denotes a sphere with six cone points  of angle $\pi$, induced by the hyperelliptic involution of $S_{2,0}$. By a result of Birman-Hilden \cite{birmanhilden}, we may realize $\Gamma_{0,6}$ as a subgroup of index 2 in $\Gamma_{2,0}$.
Along similar lines, $\Gamma_{0,5}$ is a subgroup of  index 2 in $\Gamma_{1,2}$, arising from the hyperelliptic involution of $S_{1,2}$.

\end{remark}

 The next result is the key technical observation of this note:

\begin{proposition}\label{g>3} If $g\ge3$, then for any $T<\Gamma_g$ finite,
$$\vcd (WT)+{\lambda }(T)\leq\vcd(\Gamma_g).$$
\end{proposition}
\begin{proof} Our first objective is to establish the following:

\medskip

\noindent{\bf Claim.} Let $1\ne T<\Gamma_g$ be finite, where $g\ge 3$. If  $g_T>0$, then
\begin{equation}\label{E:uno}
\vcd (WT)+\lambda (T)+1\leq\vcd(\Gamma_g).
\end{equation}

\begin{proof}[Proof of the claim]
Using \eqref{rh} and \eqref{eq:order2}, and since $g_T>0$, we have:
\begin{eqnarray}
\label{E:vcd1}
 \frac{\vcd(\Gamma_g)+1}{|T|} &=& \frac{4g-4}{|T|} = 4g_T-4+2 l_T \ge 4g_T - 4 + k_T  \nonumber \\ &=& \nu(T) \ge \vcd(WT) \nonumber.
 \end{eqnarray}
Rearranging, we obtain $ \vcd(\Gamma_g) \ge |T| \vcd(WT) - 1$. In particular, observe that equation \eqref{E:uno} is satisfied whenever
 \begin{equation}
  \label{E:vcd2}
  \vcd(WT) + \lambda (T) +2 \le |T| \vcd(WT)
 \end{equation}
 holds. We distinguish the following cases, depending on the value of $\vcd(WT)$:

\begin{enumerate}
\item If $\vcd(WT)\ge 3$, then \eqref{E:vcd1} is true for all finite subgroups $T\le \Gamma_g$, 
as $\lambda(G)\le |G|-1$ for every finite group $G$. 

\item If $\vcd(WT) =2$, then \eqref{E:vcd1} holds unless $|T|=2$, again since $\lambda(T)\le |T|-1$. But if $|T|=2$ then $\lambda(T)=1$, and thus \eqref{E:uno} follows because $\vcd(\Gamma_g) \ge 7$ as $g\ge 3$.

 \item If $\vcd(WT)=1$ then \eqref{E:vcd1} is satisfied unless $|T| \in \{2,3,4\}$. To see this, observe that if $T$ has a maximal subgroup satisfying \eqref{E:vcd1}, then the same holds for $T$, and that groups of orders $8,9,6,p$ for $p$ a prime $p>3$ satisfy \eqref{E:vcd1}. In the remaining cases $\lambda(T) \le 2$, and hence \eqref{E:uno} follows as in the previous case since $g\ge 3$. 

\end{enumerate}
This finishes the proof of the claim. \end{proof}

Returning to the proof of the proposition, let $T\le \Gamma_g$ be a finite subgroup. If $T=1$ then the result is trivial, and if $g_T>0$, then it follows from the claim above. Therefore, assume that $T\ne 1$ and $g_T =0$. Let $L\le T$ be such that $\lambda (T)=\lambda (L)+1$. Suppose first that $L=1$, noting that $\lambda(T)=1$. Since $g_L =g \ge 3$, Proposition \ref{cases} implies that $\vcd (WT)<\vcd(WL)=\vcd(\Gamma_g)$, and so we are done.
Thus assume that $L\ne 1$. If $g_L>0$, the claim above yields
\[\vcd(WT) + \lambda (T) \le \vcd(WL) + \lambda (L) +1 \le \vcd(\Gamma_g).\]  On the other hand, if $g_L=0$ then $k_T < k_L$, by Lemma \ref{L:gendec}. Thus
\begin{eqnarray}
\vcd(WT)  + \lambda (T) &=&  k_T-3 +\lambda (T)   \nonumber \\ &<& \vcd(WL) +  \lambda (L)+1 \nonumber.
 \end{eqnarray}
Hence $\vcd(WT)  + \lambda (T) \le \vcd(WL)  + \lambda (L)$, and the result follows by induction on the length of $T$.
\end{proof}

 \bigskip

We are finally ready to prove Theorem \ref{main}:

\begin{proof}[Proof of Theorem \ref{main}] 
First, if $g\ge 3$, the result follows combining L\"uck's Theorem \ref{luck} with Theorem  \ref{bredon} and Proposition \ref{g>3}. If $g=0$ then $\Gamma_{0,0}=1$ so the result is trivial. Next, if $g=1$ then $\Gamma_{1,0}=\text{SL}_2(\Z)$ and one can take the dual tree to the Farey graph as a model of $\e \Gamma_{1,0}$. For $g=2$,  Ji's spine \cite{ji} serves as a model for  $\e \Gamma_{2,0}$. This finishes the proof.
\end{proof}

\end{document}